\documentclass[preprint,12pt]{elsarticle}
\usepackage{amsfonts}
\usepackage{amssymb}
\usepackage{mathrsfs}
 \usepackage{indentfirst}
 \usepackage{graphicx}

\usepackage{subfigure}
    \usepackage{verbatim}
    \usepackage{latexsym}

  \usepackage{amsmath,amssymb,amsthm}

\journal{}

\begin{document}

\begin{frontmatter}
    \title{Potential Distribution on Random Electrical Networks}
     \author{Da-Qian Qian and Xiao-Dong Zhang\\
     Department of Mathematics, \\
     Shanghai Jiao Tong University, 200240, P.R.China\\
      xiaodong@sjtu.edu.cn}

\begin{abstract}
Let  $G=(V,E)$  be a random electronic network with the boundary
vertices which is obtained by assigning a resistance of each edge in
a random graph in $\mathbb{G}(n,p)$ and the voltages on the boundary
vertices. In this paper, we prove that  the potential distribution
of all vertices of $G$ except for the boundary vertices are very
close to a constant with high probability for $p=\frac{c\ln n}{n}$
and $c>1$.
\end{abstract}
\end{frontmatter}
\section{Introductions}
 \newtheorem{theorem}{Theorem}[section]
\newtheorem{lemma}[theorem]{Lemma}
\newtheorem{corollary}[theorem]{Corollary}
\newtheorem{remark}{Remark}
\newtheorem{definition}{Definition}
The connection between random walks and electrical networks can be
recognized by  Kakutani \cite{kakutani1945} in 1945.  Doyle and
Snell \cite{doyle1984} in 1984 published an excellent book which
explained the relations between random walks and electrical
networks.  Tetali \cite{Tetali1991} presented an interpretation of
effective resistance in electrical networks in terms of random walks
on underlying graphs.  Recently, Palacios \cite{palacios2009}
studied the hitting times on random walks on trees through electroic
networks.  For more information between random walks and electrical
networks, the readers may be referred to \cite{griffeath1982,
kemeny1966, kelly1979}.

Let $G=(V,E)$ be a connected undirected graph without loops and
multiple edges. To make it a electrical network, we assign to each
edge $e=(i,j)\in E$ resistance $r_{ij}$ and the conductance of
$e=(i,j)$ is $c_{ij}=\frac{1}{r_{ij}}$.  We define a random walk on
the electrical network $N=(G,c)$ be a Markov chain
 with transition matrix
$\textbf{P}=(p_{ij})$ given by
\[p_{ij}=\frac{c_{ij}}{c_i}\]
with $c_i=\sum_{ (i,j)\in E}{c_{ij}}$. If  $r_{ij}$  is assigned to
be a unit resistance,  it is just the simple random walk on $G$.

It is well known that the potential and current distributions  on
electrical networks follow both Kirchhoff's  and Ohm's laws.
Further,  it is proved (for example, see \cite{bollobas1998}) that
the potential distribution follows a harmonic function with certain
boundary conditions, i.e., the potential of each vertex except for
the boundary vertices is just the weight-average of the potential of
its neighbor's.  Here boundary vertices  are those vertices at which
there are current flowing into or out of the network. Moreover,
Lyons and Peres in \cite{lyons2010} investigated the potential
distributions on regular lattice. Curtis and Morrow \cite{curis1991}
determined the distribution of resistors in a rectangular network by
using boundary measurements. These works established the intimate
connection between the random walks on graphs and electrical current
networks.  It is natural to ask what we can say if the random walks
and electrical networks are considered  on a random graph?

The space $\mathbb{G}(n,p)$, is defined for $0\le p\le 1$. To get a
random element of this space, we select the edges independently,
with probability $p$.  For more information and background, the
readers are referred to Bollob\'{a}s ' book \cite{bollobas2001}.
Recently, many other  random graph models, such as small-world model
\cite{watts1998}, BA model \cite{barabasi1999} etc.,have been
proposed to simulate and study the mass real world networks .

Grimmett and Kesten

??

 considered the effective resistances for random electrical
networks on random model $\mathbb{G}(n,p)$.
 In this paper we mainly
investigate the potential distributions of the electrical network on
random graph $G\in\mathbb{G}(n,p)$ by assigning a resistance
$c_{ij}$ for each edge $e=(i,j)$. By making use of the probabilistic
interpretation of potential distribution on electrical networks  as
well as the fast mixing property \cite{cooper2007} for random walks
on a random graph $G\in \mathbb{G}(n,p)$. A sequence of events
${\mathcal{E}}_n$ is said to occur with high probability ({\bf whp})
if $\lim_{n\rightarrow \infty} Pr({\mathcal{E}}_n)=1$. The main
result of this paper can be stated as follows:

\begin{theorem}\label{mainresult}
Let $N=(G,c)$ be a random electrical network from a  random graph
$G\in \mathbb{G}(n,p)$ model with $p=\frac{\alpha\ln n}{n},
\alpha>1$ and $C_1\leq c_{ij}\leq C_2$ for all $e=(i,j)\in E$, where
$0<C_1\leq C_2 $ are two positive constants. If  $\Gamma \triangleq
\{x_1,x_2,\cdots,x_K\}$ is the boundary set with boundary potential
$V_{x_k}=p_{x_k}, 0\leq p_{x_k}\leq 1$ for $k=1,2,\cdots,K$, then
\textbf{whp} the potential distribution satisfies
\begin{equation}\label{1}
V_i=\frac{\sum_{k=1}^Kp_{x_k}c_{x_k}}{\sum_{k=1}^Kc_{x_k}}+O((\ln
n)^{-1})
\end{equation}
 for each $i\in V(G)\setminus\Gamma$.
\end{theorem}
The rest of this paper is arranged as follow. In sections 2, we
present some properties for both random graphs and electrical
networks. In section 3, we will give a rigorous proof for the main
theorem (Theorem~\ref{mainresult}) and apply this result to a
generalized consensus model with finite leaders. In section 4, we
will do some further discussions on the potential distributions in
more general cases where $G$ may be circles and small-world networks
and $c_{ij}$ may be i.i.d random variables for each $(i,j)\in E$. We
note here that we say an event holds with high probability(denoted
as \textbf{whp}) if it holds with probability $1-o(1)$ as $n\to
\infty$.
\section{Preliminaries}
In this section,  we introduce some properties of a random graph in
 $\mathbb{G}(n,p)$ and give the probabilistic interpretation of
potential distributions on electrical
 networks. Let $d(i)$ denote the degree of vertex $i \in V$ and let
 $\delta(G)$ denote the minimum degree of $G$.
\begin{definition}\label{properdef}
A simple graph $G=(V,E)$ is said to be proper if it has the
following structural properties $\textbf{P1}-\textbf{P5}$.
\quad\\
\textbf{P1}: $G$ is connected.\\
\textbf{P2}: Call a cycle short if its length is at most $\frac{\ln
n}{10\ln \ln n}$. The
 minimum distance between two short cycle is at least $\frac{\ln n}{\ln \ln n}$.\\
\textbf{P3}: $G$ has at least one triangle, at least one 5-cycle and at least one 7-cycle.\\
\textbf{P4}: Let $C_1$, $C_2$ ($C_1\leq C_2$), $K$ be positive
constants. Let $N=(G,c)$ be the electrical network on $G$  with
$C_1\leq c_{ij}\leq C_2$  for each  edge $e=(i,j)\in E$. For
$L\subset V$, $|L|\leq K$, denote by $G^{\prime}=G[V\setminus
L]\triangleq(V^{\prime},E^{\prime})$ the subgraph of $G$ induced by
$V\setminus L$. For $S\subset V\setminus L$, denote by
$E_{G^{\prime}}(S,\bar{S})$ the set of edges of $G^{\prime}$ with
one end in $S$ and the other in $\bar{S}=V\setminus (L\cup S)$. If
$|S|\leq \frac{n}{2}$, then
\begin{equation}\label{P4}
\frac{\sum_{(i,j)\in E_{G^{\prime}}(S,\bar{S})}c_{ij}}{\sum_{i\in
S}c_{i}^{\prime}}\geq \frac{C_1}{6C_2},
\end{equation}
where $c_i^{\prime}=\sum_{(i,j)\in E^{\prime}}c_{ij}$.\\
 \textbf{P5}: There exists a positive constant $\delta>0$ such
that it follows for
 any vertex $i \in V$ ,
\[\delta C\ln n<d(i)<4 C\ln n .
\]
\end{definition}

{\bf Remark}  From the definition, it looks  very rare for a graph
to be proper. But there are much many graphs to be proper. In fact,
we have the following result.
\begin{lemma}\label{random-proper}
Let $G$ be a random graph in  $\mathbb{G}(n,p)$ with
$p=\frac{\alpha\ln n}{n}$ and  a constant $\alpha>1$.
Then \textbf{whp} $G\in \mathbb{G}(n,p)$ is proper.
\end{lemma}
\begin{proof}
 It follows from Lemma~1 in \cite{cooper2007} that
\textbf{whp}  $\textbf{P1}-\textbf{P3}$ hold.  Moreover, by
Lemma~6.5.2 in \cite{durrett2006}, \textbf{whp} \textbf{P5} also
holds. For \textbf{P4}, it is a simple generalization of Lemma~1 in
\cite{cooper2007}. In fact, by Lemma~1 in \cite{cooper2007}, we have
\[\frac{|E_{G^{\prime}}(S,\bar{S})|}{d_{G^{\prime}}(S)}\geq \frac{1}{6},\]
where $d_{G^{\prime}}(S)=\sum_{i\in S}d_{G^{\prime}}(i)$ and
$d_{G^{\prime}}(i)$ is the degree of vertex $i$ in $G^{\prime}$. So
we have \textbf{whp},
\[\frac{\sum_{(i,j)\in E_{G^{\prime}}(S,\bar{S})}c_{ij}}{\sum_{i\in S}c_{i}^{\prime}}
\geq \frac{C_1|E_{G^{\prime}}(S,\bar{S})|}{\sum_{i\in
S}C_2d_{G^{\prime}}(i)}\geq \frac{C_1}{6C_2}.\] This completes the
proof.
\end{proof}

Let $N=(G,c)$ be an electrical network. For any $i\in V$ let
$W_{i,N}$ denote a random walk on $N$ which starts at vertex $i$ and
let $W_{i,N}(t)$ denote the walk generated by the first $t$ steps.
Let $X_{i,N}(t)$ be the vertex reached at step $t$ and
$P_{i,N}^{(t)}(j)=Pr(X_{i,N}(t)=j)$. If  the random walk on $N$ is
irreducible and aperiodic, let \textbf{$\pi_{N}$ }be the stationary
distribution of the random walk on $N$. We also need the following
Lemma which is a slight generalization of Lemma~2 in
\cite{cooper2007}.
\begin{lemma}\label{lemma2}
Let  $G=(V, E)$ be  proper and  $N=(G,c)$ be the electrical network
with $C_1\leq c_{i,j}\leq C_2$ for each $e=(i,j)\in E$. For a subset
$\Gamma\triangleq \{x_1,x_2,\cdots,x_K\}$ of $V$, let
$N^{\prime}=(G^{\prime},c)$ be the  induced subnetwork obtained from
$G^{\prime}=G[V\setminus\Gamma]\triangleq(V^{\prime},E^{\prime})$.
Then
 there exists a sufficiently
large constant $K_0>0$ such that for all $i,j\in V^{\prime}$ and
$t>t_0=K_0 \ln n$,
\begin{equation}\label{lemma2-1}
|P_{i,N^{\prime}}^{(t)}(j)-\pi_{N^{\prime}}(j)|=O(n^{-10}),
\end{equation}
i.e., the random walk on $N^{\prime}$ mixes in time $O(\ln n)$.

Moreover, set $\mathscr{E}=\{X_{i,N}(t)\notin \Gamma, 0\leq t\leq
2t_0\}$ be the event that the random walk on $N$ started from $i$ do
not reach the vertices in $\Gamma$ for the first $2t_0$ steps. Then
we have
\begin{equation}\label{lemma2-2}
Pr(\mathscr{E})=1-O((\ln n)^{-1})
\end{equation}
and
\begin{equation}\label{lemma2-3}
Pr(X_{i,N}(2t_0)=j|\mathscr{E})=(1+O((\ln
n)^{-1}))Pr(X_{i,N'}(2t_0)=j).
\end{equation}
\begin{proof} The Lemma and Proof are almost the same as Lemmas 2-4
in \cite{cooper2007}, in which they only consider the simple random
walk on $G$.  In order to keep the consistency of our paper,  we
will simply give a sketch proof of  Lemma~\ref{lemma2} with emphasis
on different places.

 By \textbf{P3} and \textbf{P4},  the random walk on $N^{\prime}$
 is irreducible and aperiodic
and therefore it has a unit stationary distribution
$\pi_{N{\prime}}$. By using \textbf{P4} instead of and the
isoperimetric Inequality of Lemmas2-4 in \cite{cooper2007}, it is
easy to see that (\ref{lemma2-1}) holds.

For $1\leq k\leq K$, let $N_{G^{\prime}}(x_k)$ be the neighborhood
of $x_k$ in $G^{\prime}$  and $N_{G^{\prime}}(K)=\cup_{1\leq i\leq
K} N_{G^{\prime}}(x_k)$. Let $\delta_{K}$ be the minimum degree of
vertices in $N_{G^{\prime}}(K)$. Fixing  $i\in V^{\prime}$, for
$j\in V^{\prime}$, let $\mathcal {W}_{m,j}^{x_k}$($\mathcal
{W}_{m,j}^{K}$) denote the set of walks on $N^{\prime}$ which starts
at $i$, ends at $j$, are of length $2t_0$ and which leave a vertex
in the neighborhood $N_{G^{\prime}}(x_k)$($N_{G^{\prime}}(K)$)
exactly $m$ times. Let $\mathcal {W}_m^{x_k}=\bigcup_j{\mathcal
{W}_{m,j}^{x_k}}$, $\mathcal {W}_m^{K}=\bigcup_j{\mathcal
{W}_{m,j}^{K}}$ and $W=(w_0,w_1,\cdots,w_{2t_0})\in \mathcal
{W}_m^{K}$. Set
\[\rho_W=\frac{Pr(X_{i,N}(s)=w_s,s=0,1,\cdots,2t_0)}{Pr(X_{i,N'}(s)=w_s,s=0,1,\cdots,2t_0)}.\]
Then
\[1\geq \rho_W \geq (1-\frac{C_2K}{C_1\delta_{K}})^m.\]
This is because
\begin{displaymath}
\frac{Pr(X_{i,N}(s)=w_s |
X_{i,N}(s-1)=w_{s-1})}{Pr(X_{i,N'}(s)=w_s|X_{i,N^{\prime}}(s-1)=w_{s-1})}\geq
\left\{
\begin{array}{ll}
1 & \qquad {\rm if}\ w_{s-1}\notin N_{G^{\prime}}(K) \\
\frac{c_{w_{s-1}}-C_2K}{c_{w_{s-1}}} & \qquad {\rm if} \  w_{s-1}\in
N_{G^{\prime}}(K).
\end{array} \right.
\end{displaymath}
So
\begin{eqnarray*}
\qquad \qquad Pr(\mathscr{E})&=&\sum_{m\geq 0}\sum_{W\in {\mathcal
{W}_m^{K}}}Pr(W_{i,N}(2t_0)=W)\\
&=&\sum_{m\geq 0}\sum_{W\in {\mathcal
{W}_{m}^{K}}}\rho_W Pr(W_{i,N'}(2t_0)=W)\\
\end{eqnarray*}
\begin{equation}
\quad\geq\sum_{m\geq 0}p_m^{K}(1-\frac{C_2K}{C_1\delta_{K}})^m,
\end{equation}
where
\[p_m^{K}=\sum_{W\in\mathcal
{W}_{m}^{K}}Pr(W_{i,N^{\prime}}(2t_0)=W)=Pr(W_{i,N^{\prime}}(2t_0)\in\mathcal
{W}_{m}^{K}).\]
 Now fix $j$ and write
\begin{eqnarray*}
Pr(X_{i,N}=j|\mathscr{E})&=&\sum_{m\geq 0}\sum_{W\in {\mathcal
{W}_{m,j}^{K}}}Pr(W_{i,N}(2t_0)=W)Pr(\mathscr{E})^{-1}\\
&=&\sum_{m\geq 0}\sum_{W\in {\mathcal {W}_{m,j}^{K}}}\rho_W
Pr(W_{i,N{\prime}}(2t_0)=W)Pr(\mathscr{E})^{-1}.
\end{eqnarray*}
If we set
\begin{eqnarray*}
p_{m,j}^{K}&=&\frac{Pr(W_{i,N^{\prime}}(2t_0)\in\mathcal
{W}_{m,j}^{K})}{Pr(X_{i,N'}(2t_0)=j)}\\
&=&Pr(W_{i,N'}(2t_0)\, \textrm{leaves a vertex of}\, N_{G'}(K)\,
\textrm{m times}|X_{i,N'}(2t_0)=j),
\end{eqnarray*}
then
\begin{equation}
\sum_{m\geq 0}p_{m,j}^{K}(1-\frac{C_2K}{C_1\delta_{K}})^m\leq
\frac{Pr(X_{i,N}=j|\mathscr{E})}{Pr(X_{i,N'}=j)}\leq
Pr(\mathscr{E})^{-1}.
\end{equation}
 We can get by the same method as Cooper and Frieze showed in Lemma~4 in \cite{cooper2007}
that
\[p_{0,j}^{x_k}+p_{1,j}^{x_k}+p_{2,j}^{x_k}\geq 1-O((\ln n)^{-1}),\]
where
\[p_{m,j}^{x_k}=\frac{Pr(W_{i,N{\prime}}(2t_0)\in\mathcal
{W}_{m,j}^{x_k})}{Pr(X_{i,N{\prime}}(2t_0)=j)}\qquad 0\leq m \leq 2,
1\leq k \leq K.\]
 So we have
\begin{equation}
\sum_{m=0}^{2K}p_{m,j}^K \geq 1-O((\ln n)^{-1}) \end{equation} and
\begin{equation}
\sum_{m=0}^{2K}p_{m}^K \geq 1-O((\ln n)^{-1}).
\end{equation}
Now using equations (4),(5),(6), (7) and the fact $\delta_{K}>\delta
C \ln n$ from \textbf{P5},  we have
\[Pr(\mathscr{E})\geq \sum_{m=0}^{2K}p_m^{K}(1-\frac{C_2K}{C_1\delta_{K}})^m\geq(1-\frac{C_2K}{C_1\delta_{K}})^{2K}-O((\ln n)^{-1})=1-O((\ln n)^{-1}).\]
Similarly we have \[Pr(X_{i,N}(2t_0)=j|\mathscr{E})=(1+O((\ln
n)^{-1}))Pr(X_{i,N{\prime}}(2t_0)=j).\]
\end{proof}
\end{lemma}

\begin{lemma}\label{lemma3}
Let $N=(G,c)$ be a connected electrical network. Let
$\Gamma\triangleq \{x_k,1\leq k \leq K \}\subset V$ be the distinct
boundary vertices and define for each $i\in V$
\[\tau_{x_k}(i)=min\{t|X_{i,N}(t)=x_k\}\qquad 1\leq k \leq K\]
\[P_i=\sum_{k=1}^K p_{x_k}Pr(\tau_{x_k}(i)=min(\tau_{x_s}(i),1\leq s \leq K))\]
so that $P_{x_k}=p_{x_k}$ for $1\leq k \leq K$. Then $\{P_{i},i\in
V\}$ is  the same as the distribution of potentials when $x_k$ is
set at $p_k$ for $1\leq k \leq K$. Especially when $K=2$, i.e., we
choose only two vertices $x_1,x_2$ as boundary points and set
$p_{x_1}=1, p_{x_2}=0$, then $P_i=Pr(\tau_{x_1}(i)<\tau_{x_2}(i))$
is the same as the
distribution of potentials when $x_1$ is set at $1$ and $x_2$ at $0$.\\
\begin{proof}
For a electrical network with boundary set $\Gamma$, there are no
current flow into or out of the network at vertices in
$V\setminus\Gamma$. Assume that $\{V_i, i\in V\}$ be the potential
distribution of vertices in $N$. By using Kirchhoff's current law
and Ohm's law, we have for $i\in V\setminus\Gamma$
\[\sum_{(i,j)\in E}\frac{V_i-V_j}{r_{ij}}=0,\]
which implies
\[V_i=\sum_{(i,j)\in E}\frac{c_{ij}}{c_i}V_j,\]
i.e., the potential distribution follows a harmonic function.  For
fixed $k$ and each $i\in V$,  set
\[P_i^k=Pr(\tau_{x_k}(i)=min_{1\leq s \leq K}\{\tau_{x_s}(i)\}).\]
Then
\begin{eqnarray*}
P_{x_k}^k&=&1\\
 P_{x_s}^k&=&0\qquad \textrm{for} \,\,x_s\in \Gamma,s\neq k
.\end{eqnarray*} Moreover, if we consider the very first step of the
random walk on $N$ started at $i\in V\setminus\Gamma$, then
\[P_i^k=\sum_{(i,j)\in E}p_{ij}P_j^k=\sum_{(i,j)\in E}\frac{c_{ij}}{c_i}P_j^k.\]
Since
\[P_i=\sum_{k=1}^K p_{x_k}P_i^k.\]
we can get by the superposition property that $\{P_i, i\in V\}$ is
the same as the potential distributions $\{V_i,i\in V\}$  if we set
the potential of the boundary points as $V_{x_k}=p_{x_k},
k=1,2,\cdots K$. Because they both follow a harmonic function with
the same boundary conditions.

\end{proof}

\end{lemma}
\section{Proof of the main Theorem and Remarks}
In this section, we, in fact, prove a stronger result than
Theorem~\ref{mainresult}.
\begin{theorem}\label{th31} Let $N=(G,c)$ be an electrical network with $G$ being proper
 and  $C_1\leq
c_{ij}\leq C_2$ for all $e=(i,j)\in E$. If $\Gamma \triangleq
\{x_1,x_2,\cdots,x_K\}$ is boundary vertices and  the boundary
potential as $V_{x_k}=p_{x_k}, 0\leq p_{x_k}\leq 1$ for
$k=1,2,\cdots,K$,  then the potential distribution of $v_i$  is
\[V_i=\frac{\sum_{k=1}^Kp_{x_k}c_{x_k}}{\sum_{k=1}^Kc_{x_k}}+O((\ln n)^{-1})\]
for each $i\in V\setminus\Gamma$.
\begin{proof}
 Let
$N^{\prime}=(G^{\prime},c)$  be the induced network
 obtained from $N=(G, c)$ by the induced graph $G^{\prime}=G[V\setminus \Gamma]\triangleq(V^{\prime},E^{\prime})
 $. Then by Lemma~\ref{lemma3} and equations $(1),(2)$ and $(3)$, we have that for $1\leq k\leq K$
and $i\in V^{\prime}$,
\begin{eqnarray*}
P_i^k&=&Pr(\tau_{x_k}(i)=min_{1\leq s \leq K}\{\tau_{x_s}(i)\})\\
&=&Pr(\tau_{x_k}(i)=min_{1\leq s \leq K}\{\tau_{x_s}(i)\}|\mathscr{E})Pr(\mathscr{E})\\
&+&Pr(\tau_{x_k}(i)=min_{1\leq s \leq K}\{\tau_{x_s}(i)\}|\mathscr{E}^c)(1-Pr(\mathscr{E}))\\
&=&Pr(\tau_{x_k}(i)=min_{1\leq s \leq
K}\{\tau_{x_s}(i)\}|\mathscr{E})
+O((\ln n)^{-1})\ \ \ \quad \quad  \textrm{(by  (2))}\\
&=&\sum_{j\notin \Gamma}Pr(\tau_{x_k}(i)=min_{1\leq s \leq K}\{\tau_{x_s}(i)\}|\mathscr{E},X_{i,N}(2t_0)=j)Pr(X_{i,N}(2t_0)=j|\mathscr{E})\\
&+&O((\ln n)^{-1})\ \ \ \quad \quad  \textrm{(time homogeneous and Markov Property)}\\
&=&\sum_{j\notin \Gamma }Pr(\tau_{x_k}(j)=min_{1\leq s \leq
K}\{\tau_{x_s}(j)\})Pr(X_{i,N}(2t_0)=j|\mathscr{E})+O((\ln n)^{-1})
\ \ \ \textrm{}
\\
&=&(1+O((\ln n)^{-1}))\sum_{j\notin \Gamma}
P_j^kPr(X_{i,N'}(2t_0)=j)+O((\ln n)^{-1}) \ \ \ \textrm{(by  (3))}
\\
&=&(1+O((\ln n)^{-1}))\sum_{j\notin \Gamma}P_j^k\pi_{N'}(j)+O((\ln
n)^{-1}) \ \ \ \quad \quad \quad  \textrm{(by (1))}
\\
&=&\sum_{j\notin \Gamma}P_j^k\pi_{N'}(j)+O((\ln n)^{-1}).\ \quad
\quad \quad \textrm{(by} \quad P_j^k<1)
\end{eqnarray*}
Hence
\[V_i=P_i=\sum_{k=1}^K p_{x_k}P_i^k\triangleq V_c+O((\ln n)^{-1}).\]
 Since the  total current
flowing into the network  is equal to the current flowing out, we
have
\[\sum_{k=1}^K\sum_{(x_k,i)\in E}(V_{x_k}-V_i)c_{ix_k}=0.\]
Then
\[\sum_{k=1}^K\sum_{(x_k,i)\in E,i\notin\Gamma }(p_{x_k}-V_c-O((\ln
n)^{-1}))c_{ix_k}=O(1),\] which implies
\[\sum_{k=1}^K\sum_{(x_k,i)\in
E }(p_{x_k}-V_c-O((\ln n)^{-1}))c_{ix_k}=O(1).\] Hence
\[\sum_{k=1}^K(p_{x_k}-V_c)c_{x_k}=O(1).\] Using \textbf{P5}, we have
\[V_c=\frac{\sum_{k=1}^Kp_{x_k}c_{x_k}}{\sum_{k=1}^K
c_{x_k}}+O((\ln n)^{-1}).\] This completed the proof.
\end{proof}
\end{theorem}
Let us consider a special case of theorem~\ref{th31}, in which we
assign unit conductance for each edge and choose exactly two
vertices as the boundary set.
\begin{corollary} \label{cor32}
Let $N=(G,c)$ be an electrical network with $G$ being proper and
$c_{ij}$  being unit conductance. If  $\{x,y\}$ is the boundary set
and  is added to their unit potential difference as $V_{x}=1$ and
$V_y=0$, then the potential distribution of $v_i$ is
\[V_i=\frac{d(x)}{d(x)+d(y)}+O((\ln n)^{-1})\]
for each $i\in V\setminus\{x,y\}$.
\end{corollary}
\quad\\\textbf{Proof of Theorem~\ref{mainresult} :} It is easy to
see  from Lemma~\ref{lemma2} and Theorem~\ref{th31} that
Theorem~\ref{mainresult} holds.

 {\bf Remark 1} Let us now see the concentration of potential
distribution  from a different point of view. We consider a
generalized consensus model on $G\in \mathbb{G}(n,p)$ with $K$
leaders $\Gamma\triangleq \{x_k, 1\leq k\leq K\}$. For each $i\in V$
let $s_i(0)$ denote the score of the $i$th agent towards some event
at the initial state. Set $s_{x_k}(0)=p_k, 1\leq k\leq K$, at each
step all agents except for the leaders change their scores by simply
averaging the scores of their neighbors. Let
$\textbf{s}(t)=\{s_1(t),s_2(t),\cdots,s_n(t)\}^T$ be the score
vector at step t. Then
\[\textbf{s}(t+1)=\textbf{P}\textbf{s}(t),\]
where $\textbf{P}$ is just the transition matrix of a random walk on
$G$ with $\Gamma$ as absorbing states. It is known that  that
$\textbf{s}(t)$ will approach to a vector
$\textbf{s}(\infty)=\{s_1(\infty),s_2(\infty),\cdots,s_n(\infty)\}^T$
as $t\to \infty$. Moreover $\textbf{s}(\infty)$ follows a harmonic
function with the same boundary conditions as the potential
distribution on the electrical network  in Theorem~\ref{th31} when
we set unit resistance for each edge. So the limiting score vector
also has the concentration property while we choose connecting
probability $p$ large enough since the harmonic function has unit
solution.

{\bf Remark 2} In Theorem~\ref{mainresult}, in order to make $G$ be
connected, we have to set the probability $p$ larger than $\frac{\ln
n}{n}$. Otherwise $G$ may be not connected (see
\cite{bollobas2001}). But We can still  consider our model on the
giant component of $G$ for $\frac{1}{n}<p<\frac{\ln n}{n}$ (see
\cite{bollobas2001}).

However while $p$ is small, there will be many vertices on the tree
tops which are meaningless for our model. In order to avoid this we
may consider our model on a special case of small-world network . We
can get a connected graph $G$ by simply adding a random graph
$\mathbb{G}(n,p)$ to a circle. In this case while $p>\frac{\ln
n}{n}$ we can still get the concentration property by the same
method we used in theorem 3.1 since $G$ is also proper \textbf{whp}.
However while $p$ is small we are not able to give a rigorous result
now. In the section below we will do some simulations on small-world
networks while the connecting probability is small.

\section{Further Discussions and Problems}
In this paper,  we present the potential  distributions of an
electrical network on proper graphs and the resistance on each edge
being bounds. It is natural to ask what the potential distributions
on other graphs and different resistance. In this section, we
consider the potential distributions of the electrical networks on
with different graphs, such as circles, and the small-world networks
(see \cite{watts1998}) and  the resistance $c_{ij}$ be i.i.d random
variables for each $(i,j)\in E$ which t may be closed to 0 or
$+\infty$.  Up to now, there is no theoretical results as
Theorem~\ref{mainresult}, since there seems no methods to deal with
these problems. But the simulations on these questions may appeal
some ideas.

First, we note here if the potential distribution except for the
boundary vertices are very close to a constant $V_c$, then similarly
as we proved in theorem~\ref{mainresult},
\[V_c\sim \frac{\sum_{k=1}^K p_{x_k}c_{x_k}}{\sum_{k=1}^K c_{x_k}}\triangleq \bar{V}_c\]
where $\{x_k,1\leq k \leq K\}$ are boundary vertices with
$V_{x_k}=p_{x_k}$. So we can use $\bar{V}_c$ as an approximation of
$V_c$.

We divide our simulations into three parts according to the
structures of  networks and three different independently random
distributions.
\\
\\
{\bf Case 1:} Circle. Let $G$ be a circle on $1000$ vertices, each
vertex $v_i$ has exactly two neighbors
 and $v_1$ is connected with $v_{1000}$. Set the boundary potential as $V_1=1,V_{251}=0.3,V_{501}=0.7,V_{751}=1$.
  In Figure 1, we plot three pictures of potential distributions
  according to choices of conductance,  where $c_{ij}$ is unit conductance,
    unit $U(0,1)$ distribution, and
  power-law distribution (see\cite{barabasi1999}), respectively. Here we use power-law
  distribution with density function as
  \begin{equation}
  f(x)\sim x^{-2.5} \qquad \qquad \qquad for\qquad x\geq1
  .\end{equation}

\begin{figure}[h!]
  \centering
  \subfigure[$c_{ij}$ is unit conductance]{
    \label{fig:subfig:a} 
    \includegraphics[width=1.6in]{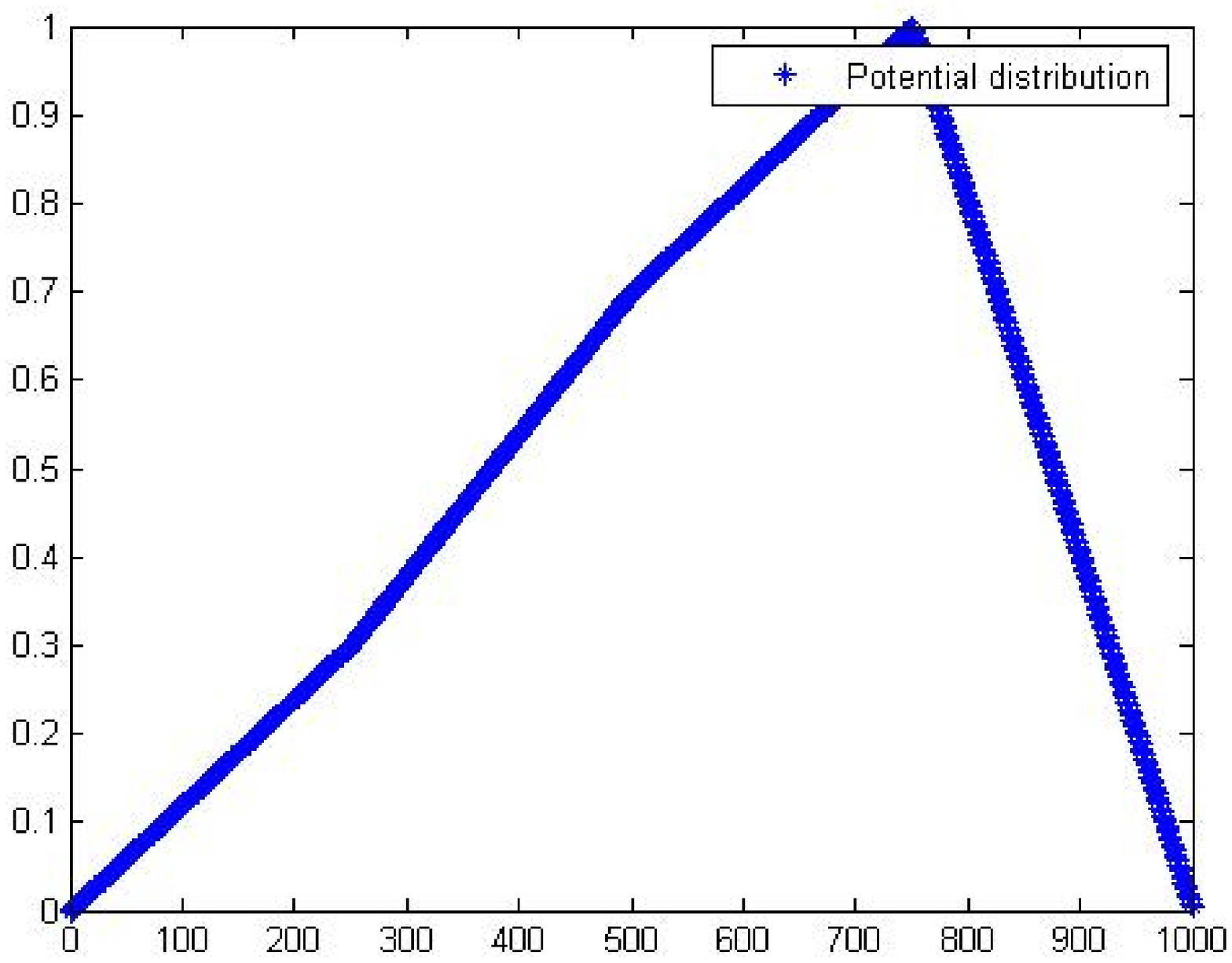}}
  \hspace{0.1in}
  \subfigure[$c_{ij}$ follows $U(0,1)$ distribution]{
    \label{fig:subfig:b} 
    \includegraphics[width=1.6in]{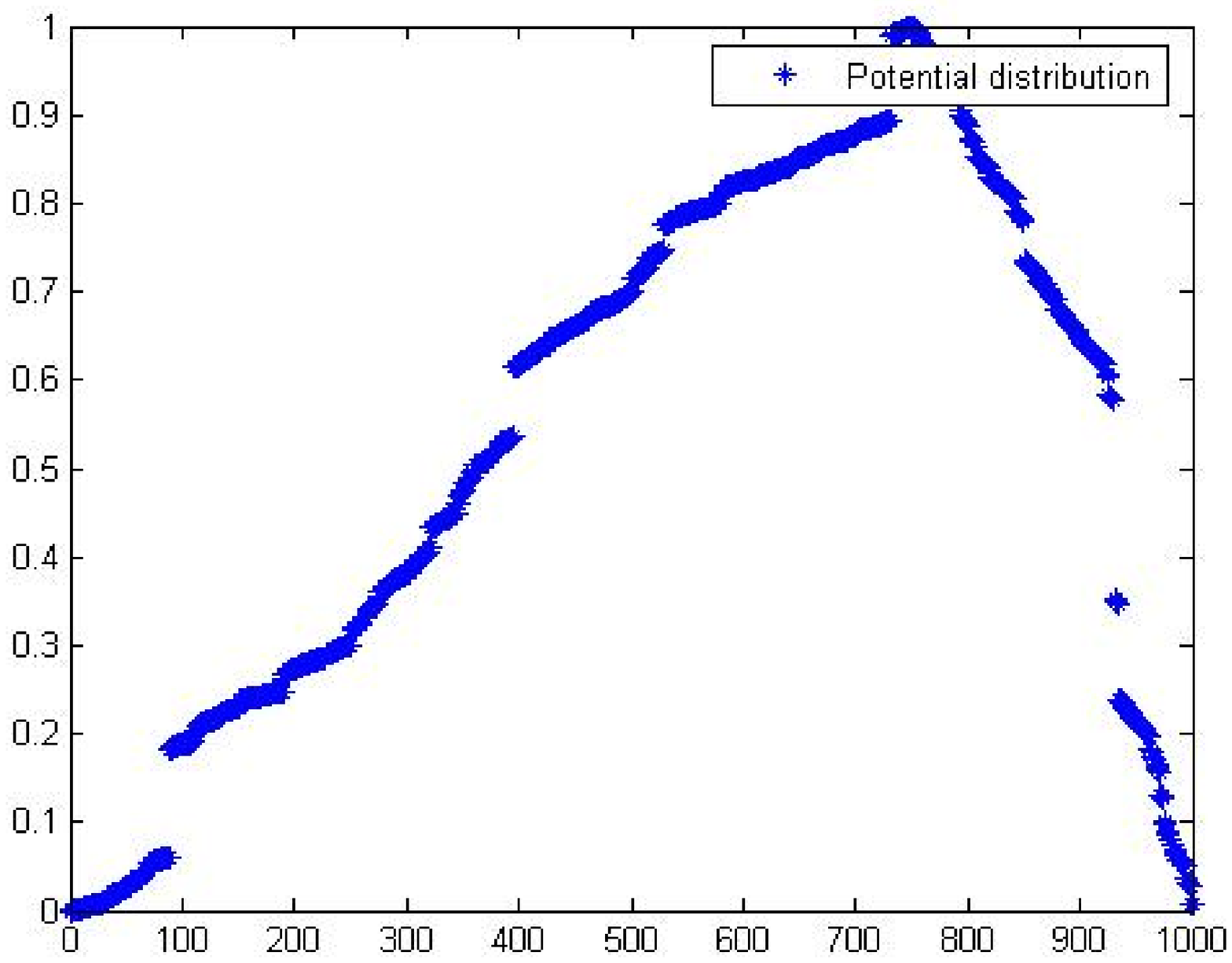}}
    \hspace{0.1in}
  \subfigure[$c_{ij}$ follows power-law distribution]{
    \label{fig:subfig:b} 
    \includegraphics[width=1.6in]{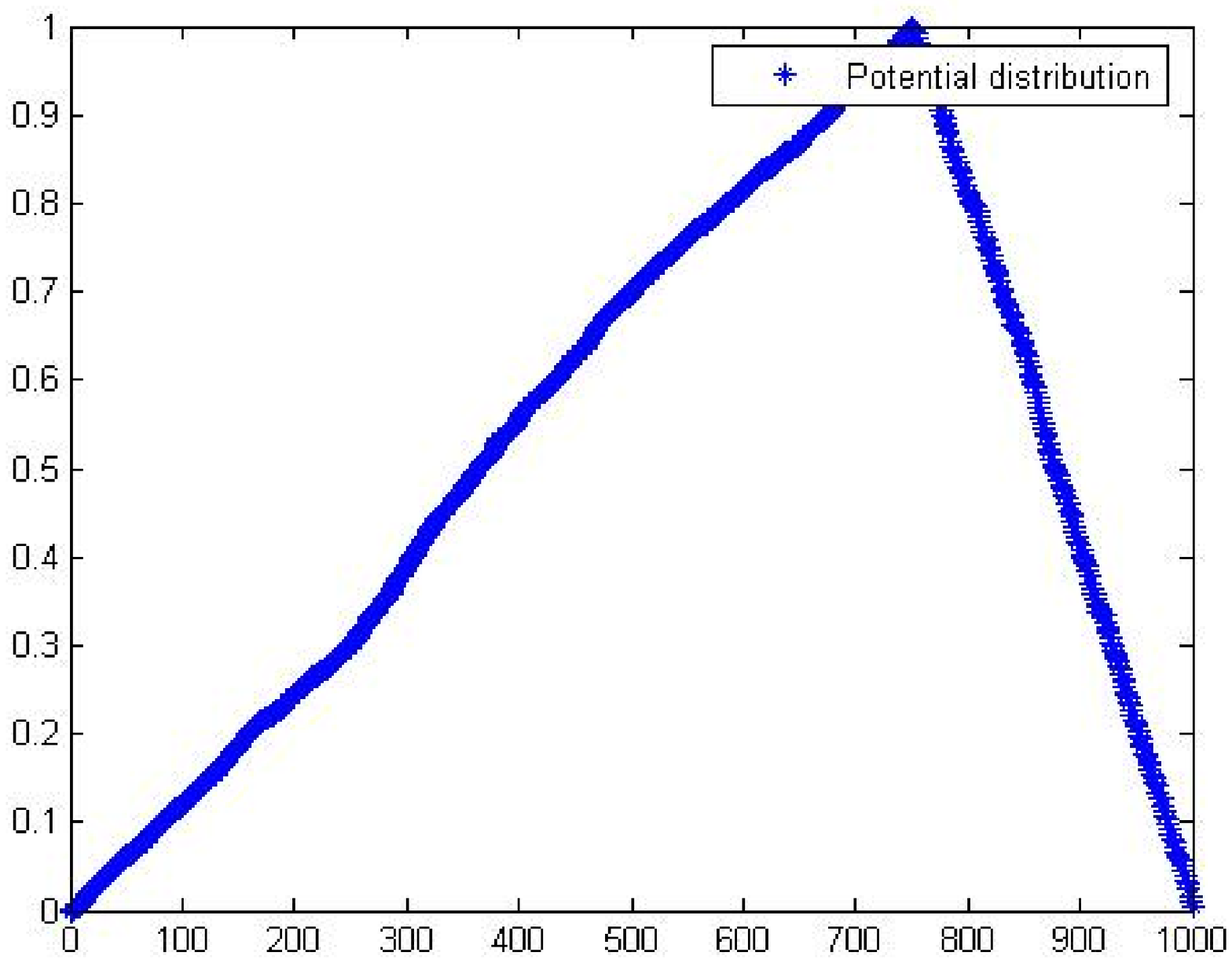}}
  \caption{Potential distribution on  circles}
  \label{fig:subfig} 
\end{figure}
\quad
  From Figure 1, it is easy to see that there exist no concentration of
potential distributions on circles no matter how we choose any
distribution of conductance.

{\bf Case 2:}  $\mathbb{G}(n,p)$ model. We choose $G \in
\mathbb{G}(n,p)$ with $n=1000$ , $p=0.01$ and the expected average
degree is $10$. So $G$ is proper \textbf{whp}. Set the boundary
potential as $V_1=1,V_{251}=0.3,V_{501}=0.7,V_{751}=1$. In Figure 2,
we plot three pictures of potential distributions on $N=(G, c)$,
where $c$ is unit conductance,
     unit $U(0,1)$ distribution, and
  power-law distribution (see\cite{barabasi1999}), respectively.

\begin{figure}[h!]
  \centering
  \subfigure[$c_{ij}$ is unit conductance]{
    \label{fig:subfig:a} 
    \includegraphics[width=1.6in]{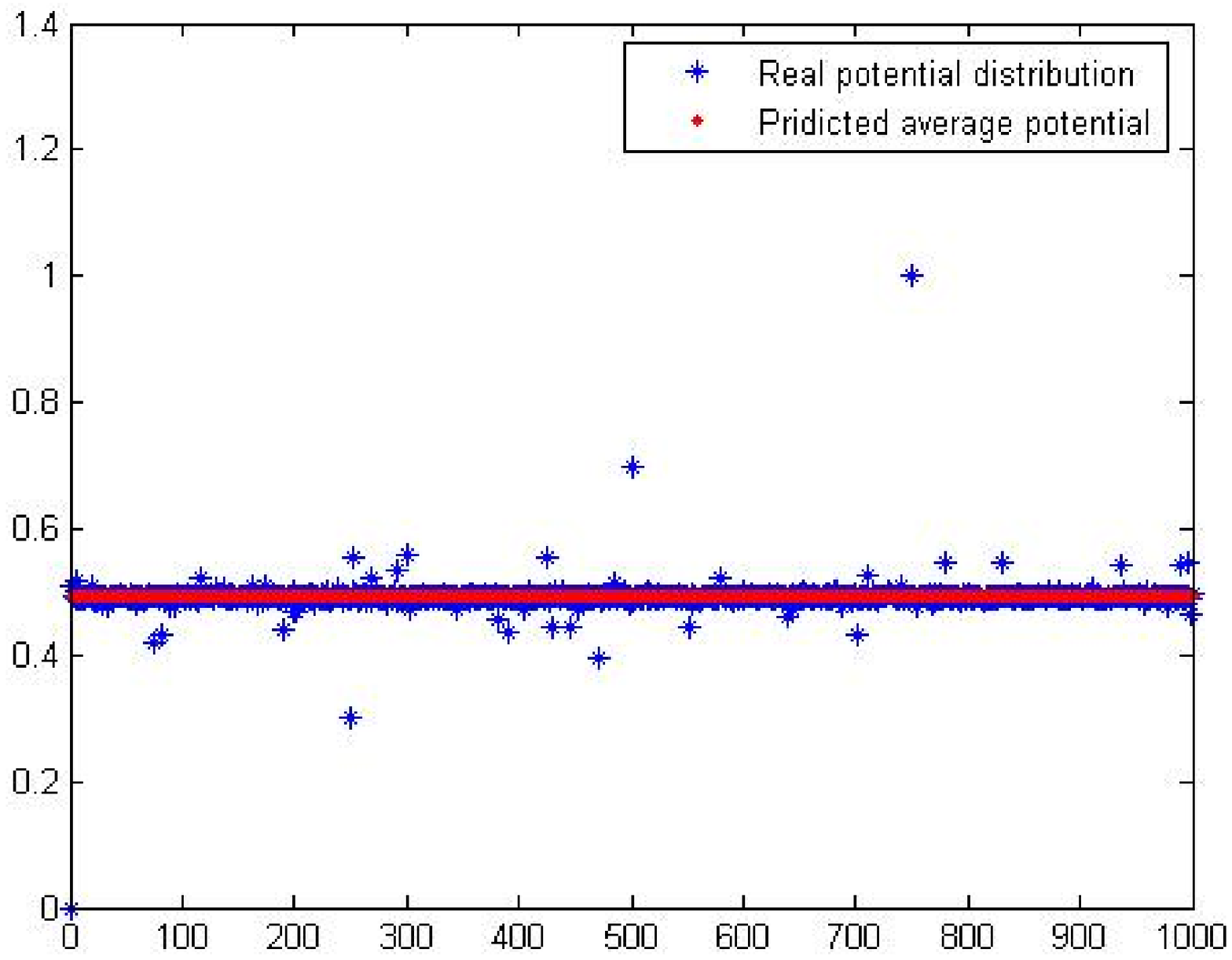}}
  \hspace{0.1in}
  \subfigure[$c_{ij}$ follows $U(0,1)$ distribution]{
    \label{fig:subfig:b} 
    \includegraphics[width=1.6in]{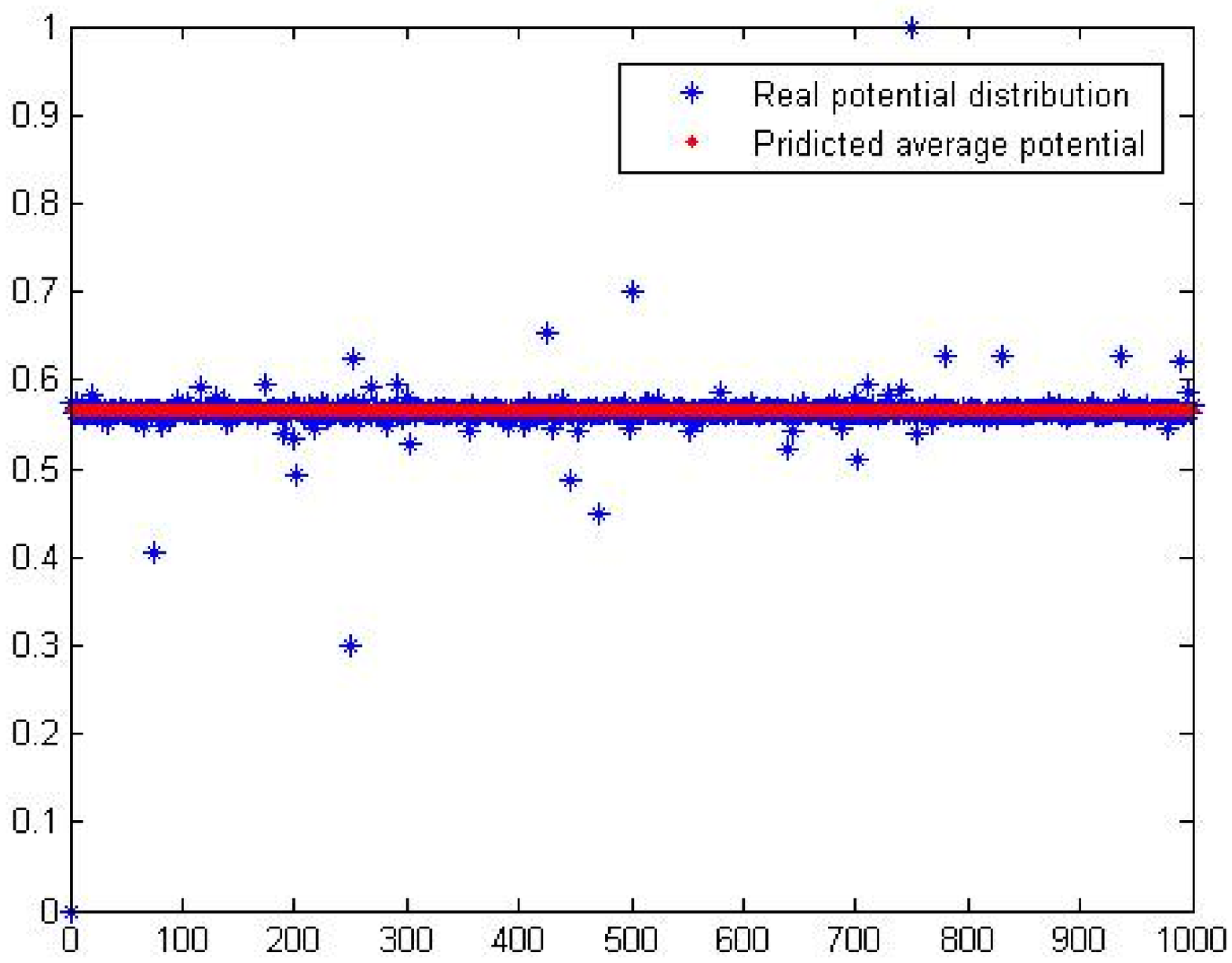}}
     \hspace{0.1in}
  \subfigure[$c_{ij}$ follows power-law distribution]{
    \label{fig:subfig:b} 
    \includegraphics[width=1.6in]{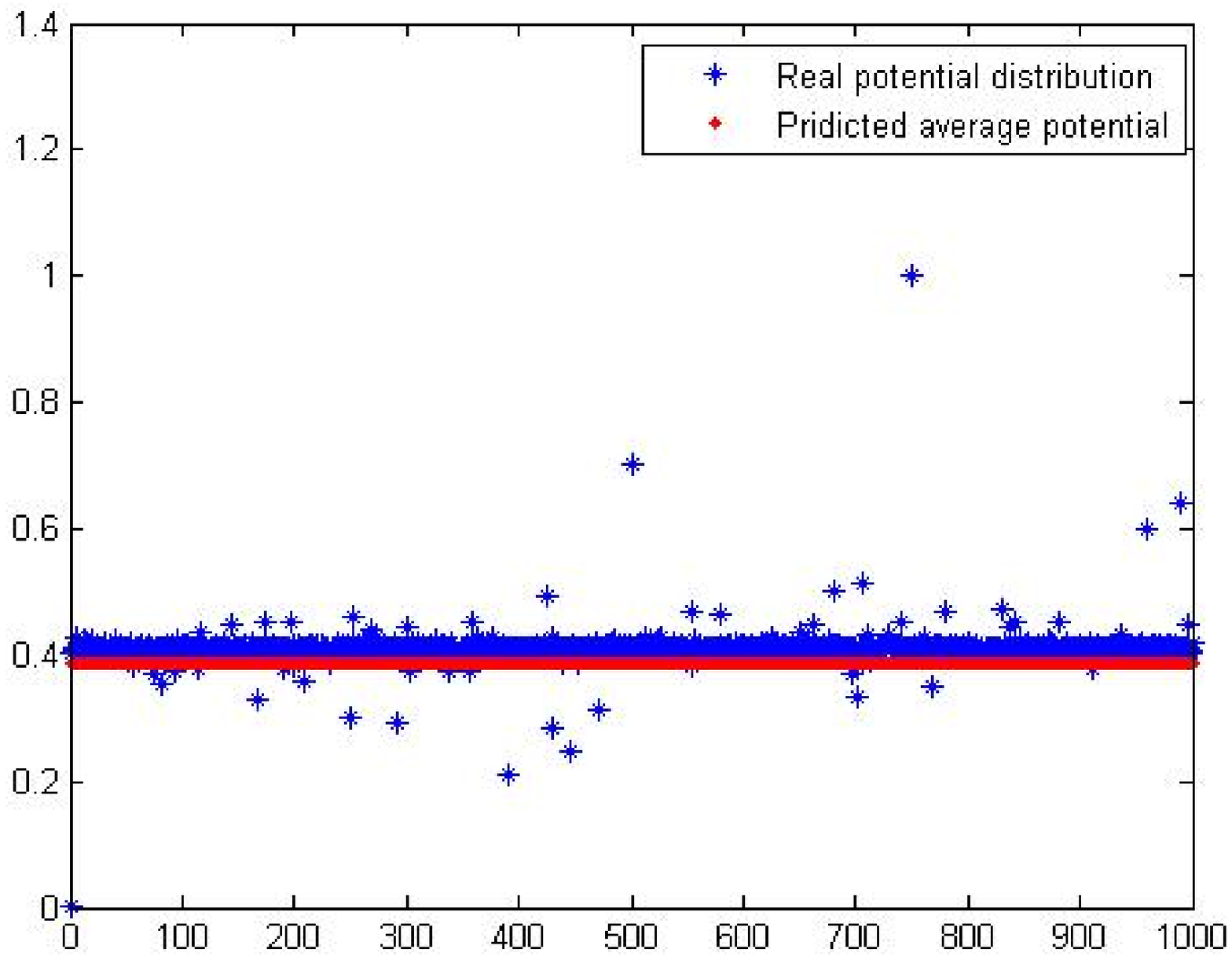}}
  \caption{Potential distribution on $\mathbb{G}(n,p)$ graphs}
  \label{fig:subfig} 
\end{figure}
\quad
\\
\\

  From Figure 2, it is easy to see that concentration of potential
distribution appears when we choose unit conductance as we proved in
Theorem~\ref{mainresult}, and we can use $\bar{V_c}$ as an efficient
approximation of $V_c$.  Even if  the conductance  follow certain
distributions such that it  approaches to $0$ or $+\infty$,  we can
still find concentration properties. But we are not able to give a
rigorous mathematical proof in this case.
\\
{\bf Case 3:} The small world network. We choose $G$ to be a random
graph $\mathbb{G}(n,p)$ adding to a
 circle of size $n$. Here we choose $n=1000$, $p=0.001$  so that $G$ may not be proper.
 Set the boundary potential as $V_1=1,V_{251}=0.3,V_{501}=0.7,V_{751}=1$. In Figure 3, we plot three pictures of potential
distributions, where $c$ is unit conductance,
     unit $U(0,1)$ distribution, and
  power-law distribution (see\cite{barabasi1999}), respectively.

 From Figure 3, it is easy to see  that even if we choose the
connecting probability $p$ very small, there also exists a
concentration of potential distribution on small-world network
except for a few vertices. We guess this is because the random walk
on small-world network also has short mixing time  as the random
walk on $\mathbb{G}(n,p)$ model with large $p$ which we mentioned in
equation (1).

It seems from the simulation results that only the structure of the
electrical network will affect the concentration property. So we
propose the following two questions:\\
{\bf Problem 4.1} Does the potential distribution for an electrical
network $N=(G,c)$ concentrate where $G$ is proper and $c_{i,j}$ is
i.i.d?
\\
{\bf Problem 4.2} Does the potential distribution for a electrical
network $N=(G,c)$ concentrate where $G$ is from small-world network
with low connecting probability and $c_{i,j}$ is unit conductance or
i.i.d as unit $U(0,1)$ or power-law distributions?\\

\begin{figure}[h!]
  \centering
  \subfigure[$c_{ij}$ is unit conductance]{
    \label{fig:subfig:a} 
    \includegraphics[width=1.6in]{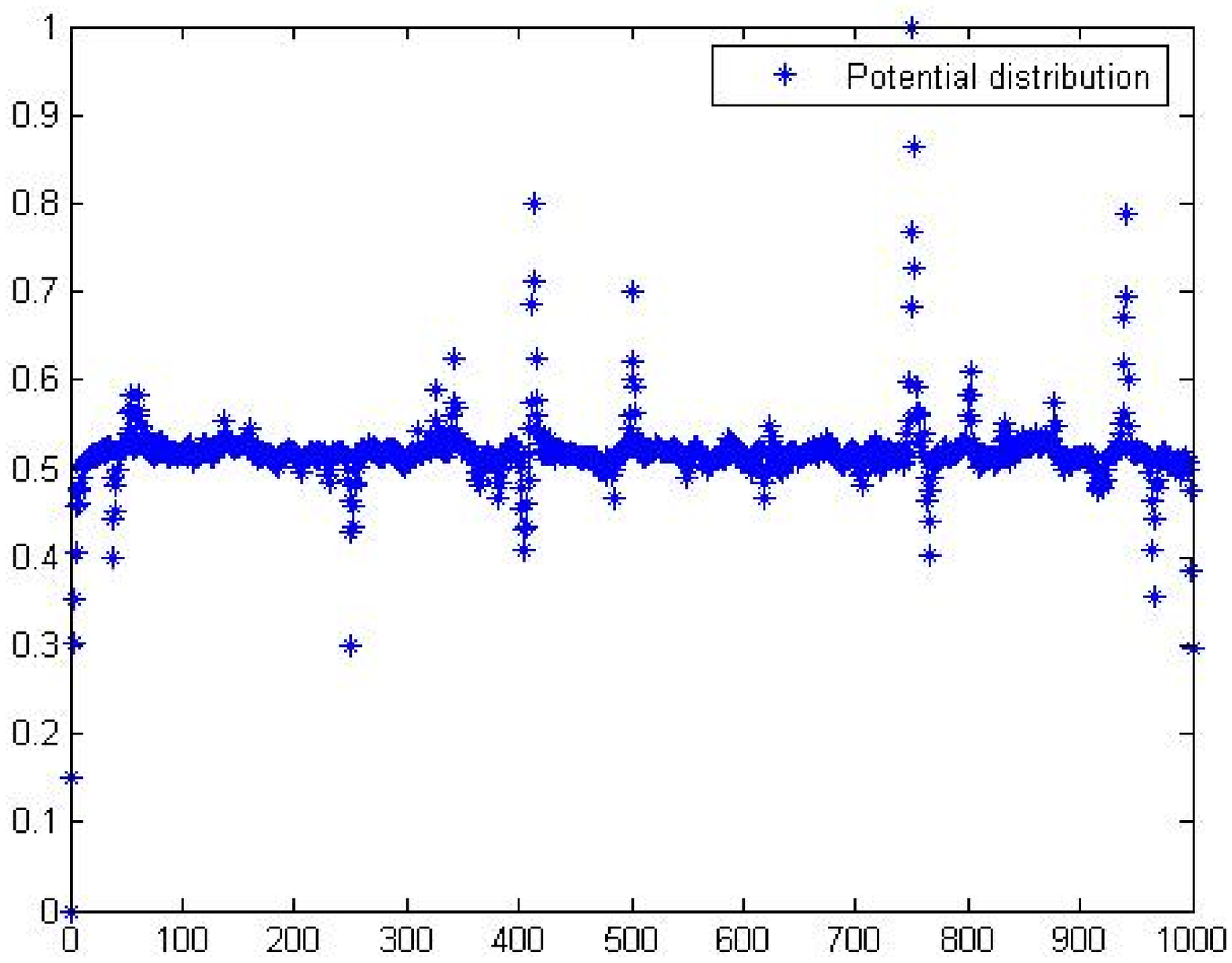}}
  \hspace{0.1in}
  \subfigure[$c_{ij}$ follows $U(0,1)$ distribution]{
    \label{fig:subfig:b} 
    \includegraphics[width=1.6in]{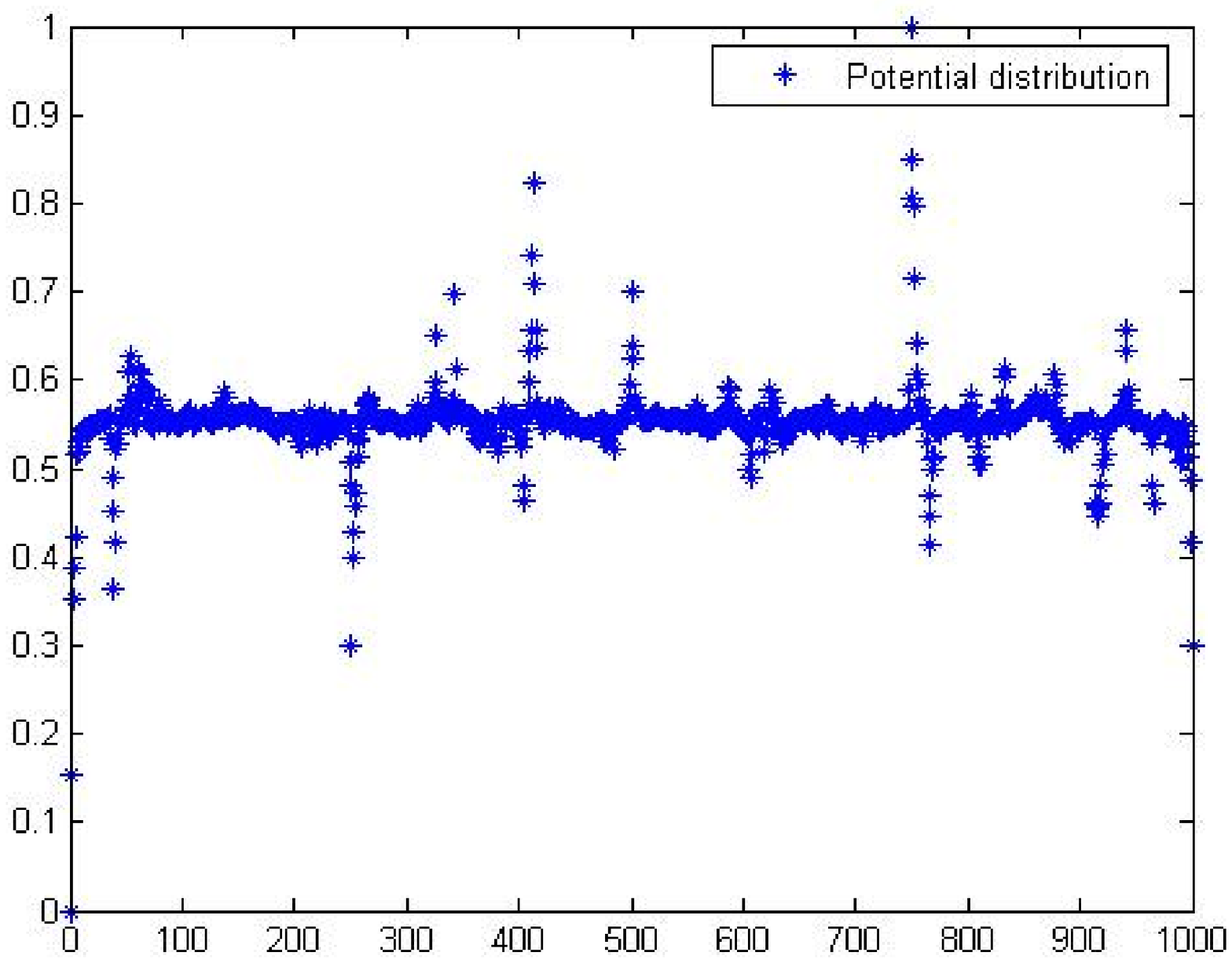}}
 \hspace{0.1in}
  \subfigure[$c_{ij}$ follows power-law distribution]{
    \label{fig:subfig:b} 
    \includegraphics[width=1.6in]{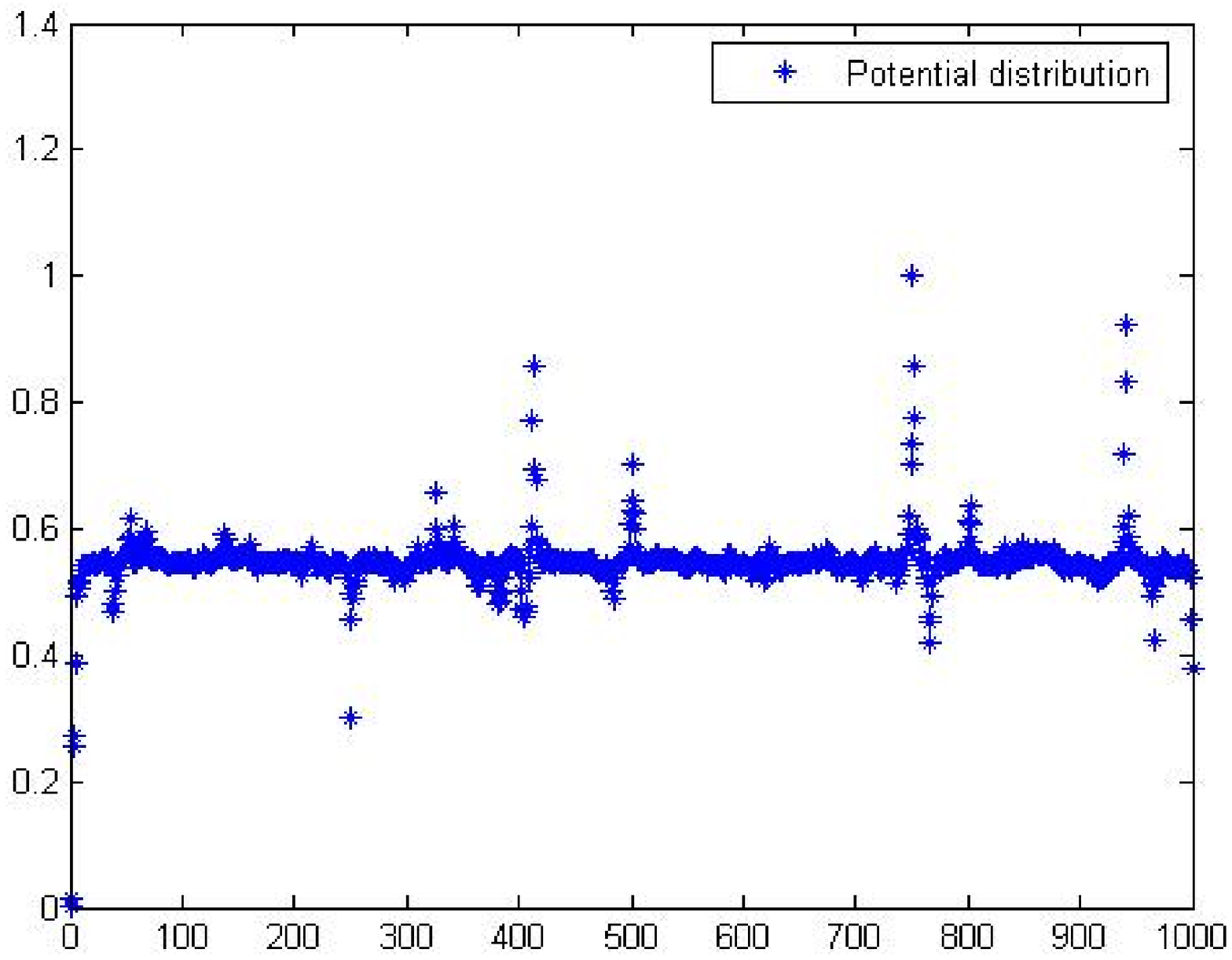}}
  \caption{Potential distribution on small-world networks}
  \label{fig:subfig} 
\end{figure}

\end{document}